\documentclass[preprint,12pt]{elsarticle}

\usepackage{amsmath,amssymb,amsthm,amsfonts,bm}
\usepackage{enumitem}
\usepackage[numbers]{natbib}
\usepackage{tabularx}
\usepackage{array}
\usepackage{placeins}
\usepackage{tikz}          
\usetikzlibrary{3d}        
\usetikzlibrary{arrows.meta} 

\newcolumntype{Y}{>{\raggedright\arraybackslash}X}

\theoremstyle{definition}
\newtheorem{definition}{Definition}[section]
\newtheorem{remark}{Remark}[section]
\newtheorem{lemma}{Lemma}[section]
\newtheorem{theorem}{Theorem}[section]
\newtheorem{proposition}{Proposition}[section]
\newtheorem{corollary}[theorem]{Corollary}
\newtheorem{example}[theorem]{Example}

\begin{document}
	
	\begin{frontmatter}
		
	\title{A Parametric Theory of Vector Spaces with Applications to Fuzzy Vector Spaces}
		
		\author[1]{Bayaz Daraby}
		\ead{daraby@maragheh.ac.ir}

		\author[1]{Hasan Haddadzadeh\corref{cor1}}
		\ead{Mh.hadadzadeh@stu.maragheh.ac.ir}
		\ead{m_h_haddad@yahoo.com}
		
		\cortext[cor1]{Corresponding author}
		
		\address[1]{Department of Mathematics, University of Maragheh,
			P.O. Box 55136--553, Maragheh, Iran}
	\begin{abstract}
In this paper, we propose a parameterized framework for vector spaces in which each vector is represented by a family of parameter-dependent realizations.
	
	\end{abstract}

		
	\end{frontmatter}
	
	\allowdisplaybreaks

\section{Introduction} Since the introduction of fuzzy set theory by Zadeh~\cite{zadeh1965fuzzy}, the interaction between fuzziness and linear algebra has become one of the central topics in fuzzy mathematics. This interaction has led to the development of fuzzy vector spaces, fuzzy linear transformations, fuzzy subspaces, fuzzy normed spaces, and fuzzy topological vector spaces, which now constitute an important part of fuzzy algebra and functional analysis. Among the earliest and most influential models is the fuzzy vector space introduced by Lubczonok~\cite{Lubczonok1990}. In this framework, a fuzzy vector space is a pair $(V,\nu)$, where $V$ is a vector space and $\nu:V\rightarrow[0,1]$ satisfies \[ \nu(ax+by)\ge \nu(x)\wedge \nu(y), \qquad a,b\in\mathbb F,\;x,y\in V. \] Here fuzziness is represented by assigning a membership degree to each vector of the ambient space. This viewpoint has inspired numerous extensions and generalizations involving algebraic, topological, and categorical aspects of fuzzy vector spaces. Recent developments include applications to coding theory, convexity, and generalized algebraic structures \cite{Bejines2024,Li2024,Gereme2024}. Despite their diversity, these models share the common feature that fuzziness is encoded through a membership function on vectors. The philosophy adopted in the present paper is fundamentally different. Instead of assigning a membership degree to each vector, every crisp vector \[ x\in X \] is associated with an entire parameterized family \[ \widetilde{x}=\{x_\alpha\}_{\alpha\in(0,1]}, \] called a \emph{fuzzy labeling} of $x$. The parameter $\alpha$ is not interpreted as a membership value but as an index identifying a representative of the underlying vector. Consequently, fuzziness is encoded through the evolution of these representatives rather than by a fuzzy subset of the ambient space. This construction naturally gives rise to two families of vector spaces. For each $\alpha\in(0,1]$, the representatives generate the level space \[ X_\alpha=\operatorname{span}\{x_\alpha:x\in X\}, \] and the tail space \[ \widetilde X_\alpha = \operatorname{span}\{x_\beta:\beta\ge\alpha,\;x\in X\}. \] The family \[ \{\widetilde X_\alpha\}_{\alpha\in(0,1]} \] forms a nested filtration of vector spaces that records the progressive appearance of new linear directions as the parameter decreases. This filtration contains structural information about the labeling process that has no counterpart in classical membership-based models. To describe this structure, we introduce several algebraic invariants, including the dimension profile \[ \delta(\alpha)=\dim(\widetilde X_\alpha) \] and the persistence level of subspaces. These invariants measure how the filtration evolves and are shown to be preserved under strictly increasing reparameterizations, providing natural tools for the classification of fuzzy labelings. Although the proposed framework is not defined through a membership function, it admits a natural connection with the classical theory. We prove that every fuzzy labeling canonically induces a Lubczonok fuzzy vector space through the persistence function associated with the tail-space filtration. Conversely, under suitable finite-dimensional assumptions, every Lubczonok fuzzy vector space admits a representation by an appropriate fuzzy labeling. This yields a representation theorem linking the two approaches. The correspondence, however, is not one-to-one. Different fuzzy labelings may induce the same Lubczonok fuzzy vector space, showing that classical membership-based models classify fuzzy labelings only up to equivalence. In particular, the labeling mechanism preserves structural information that is lost when only the induced membership function is retained. Another distinctive feature of the labeling framework is the intrinsic order carried by the representatives, providing each generated vector with a unique position in the labeling process. This ordered structure disappears after passing to the associated Lubczonok fuzzy vector space.
\section{Preliminaries}

In this section, we briefly recall several basic concepts from fuzzy set theory and fix the notation used throughout the paper.
Let $\mathbb{R}$ denote the set of real numbers.

\begin{definition} \cite{zadeh1965fuzzy} A \emph{fuzzy set} on $\mathbb{R}$ is a mapping \[ u:\mathbb{R}\rightarrow[0,1]. \] For $\alpha\in(0,1]$, the corresponding $\alpha$-level set (or $\alpha$-cut) of $u$ is defined by \[ [u]_\alpha=	
	\{x\in\mathbb{R}:u(x)\ge\alpha\}. \] \end{definition}
\begin{definition}A fuzzy set $u:\mathbb{R}\rightarrow[0,1]$ is called a \emph{fuzzy number} if
	\begin{enumerate} \item $u$ is normal; \item $u$ is fuzzy convex; \item $u$ is upper semicontinuous; \item $\operatorname{supp}(u)$ is compact. \end{enumerate}
	
	Equivalently, for every $\alpha\in(0,1]$, its $\alpha$-cut is a compact interval \[ [u]_\alpha=
	[\underline{u}(\alpha),\overline{u}(\alpha)], \] where $\underline{u}$ is nondecreasing and $\overline{u}$ is nonincreasing. \end{definition}
The representation of fuzzy numbers through their $\alpha$-cuts provides a convenient parametric description that has become one of the fundamental tools of fuzzy analysis. In this representation, a fuzzy number is completely determined by the family of intervals \[ {[\underline{u}(\alpha),\overline{u}(\alpha)]}_{\alpha\in(0,1]}, \] subject to the monotonicity and continuity conditions of the endpoint functions.
Throughout this paper, let $V$ be a vector space over $\mathbb{R}$ (or $\mathbb{C}$) and let $X$ be a vector subspace of $V$. Unless otherwise stated, all vector space operations are understood in the usual sense.\\
The main idea of the present work is inspired by the parametric philosophy of fuzzy numbers. Instead of associating with an uncertain scalar an $\alpha$-indexed family of interval endpoints, we associate with each crisp vector an $\alpha$-indexed family of vector representatives. This viewpoint serves as the basis for the definition of fuzzy vectors introduced in the next section and provides the foundation for the construction of fuzzy inner products and fuzzy norms developed throughout the paper.
\begin{definition}\cite{Lubczonok1990}
	Let \(V\) be a vector space over a field \(\mathbb{F}\).
	A \emph{fuzzy vector space} in the sense of Lubczonok is a pair
	\[
	(V,\nu),
	\]
	where
	\[
	\nu:V\longrightarrow[0,1]
	\]
	is a membership function satisfying
	\[
	\nu(ax+by)\ge
	\nu(x)\wedge\nu(y),
	\qquad
	\forall\,x,y\in V,\;
	\forall\,a,b\in\mathbb{F}.
	\]
	Here \(\wedge\) denotes the minimum operation on \([0,1]\).

\end{definition}
\begin{proposition}\cite{Lubczonok1990}
	If \(E=	(V,\nu)\) is a fuzzy vector space and if \(v,w\in V\) with
	\[
	\nu(v)\neq\nu(w),
	\]
	then
	\[
	\nu(v+w)=\nu(v)\wedge\nu(w).
	\]
\end{proposition}
۱\begin{definition}\cite{Ghrist2014}
	Let \(V\) be a vector space.
	A family of subspaces
	\[
	\{V_\alpha\}_{\alpha\in(0,1]}
	\]
	is called a \emph{filtration} of \(V\) if
	\[
	\alpha_1<\alpha_2
	\quad\Longrightarrow\quad
	V_{\alpha_2}\subseteq V_{\alpha_1}.
	\]
	
	If, in addition,
	\[
	V=\bigcup_{\alpha\in(0,1]}V_\alpha,
	\]
	then the filtration is called \emph{exhaustive}.
\end{definition}
\section{Parametric Vector}
We begin by introducing the concept of a fuzzy vector, which forms the basis of the results and constructions presented in the remainder of this paper.

\begin{definition}\label{def:fuzzyvector}
	Let $V$ be a vector space over $\mathbb{R}$ (or $\mathbb{C}$), and let $X$ be a vector subspace of $V$. A fuzzy labeling on $X$ is a mapping
	\[
	\mu : (0,1] \times X \longrightarrow V, \qquad (\alpha, x) \mapsto x_\alpha,
	\]
	satisfying
	\[
	x_1 = x, \quad \forall x \in X.
	\]
	For every $x \in X$ and every $\alpha \in (0,1]$, define the $\alpha$-tail fuzzy vector associated with $x$ by
	\begin{equation}\label{eq:alpha-tailfuzzyvector}
     \widetilde{x}_{\alpha}
	:= \{x_{\beta}\}_{\beta \in [\alpha,1]}
	= \{x_{\beta} : \beta \in [\alpha,1]\}.
		\end{equation}
	The family $\widetilde{x}_{\alpha}$ is called the $\alpha$-level tail fuzzy vector associated with $x$.\\
	For every $x \in X$, the family
	\begin{equation}\label{eq:fuzzyvectorassociated-x}
     \widetilde{x} := \{x_{\beta}\}_{\beta \in (0,1]}
     \end{equation}
	is called the fuzzy vector associated with $x$, and $x = x_1$ is called its underlying crisp vector. The vector $x_\alpha$ is called the $\alpha$-level representative of $\widetilde{x}$.\\
	For every $\alpha \in (0,1]$, define
	\begin{equation}\label{vectorspacegenerated-alpha-levelrepresentatives}
	X_{\alpha} := \operatorname{span}\{x_{\alpha} : x \in X\},
	\end{equation}
	that is, the vector subspace generated by all $\alpha$-level representatives,
	and
\begin{equation}\label{vectorsubspacegenerated-alpha-levelfuzzytail}
    \widetilde{X}_{\alpha} := \operatorname{span}\{x_{\beta} : x \in X,\ \beta \in [\alpha,1]\}.
\end{equation}
	We assume that
	\[
	X_{1} =	\widetilde{X}_{1} = X.
	\]
	
\end{definition}

 \begin{remark}
 	\begin{enumerate} \item The space $X_\alpha$ consists precisely of the vectors generated by the representatives at the fixed level $\alpha$. In the subsequent sections, the fuzzy inner product and the induced fuzzy norm will be defined on the family of spaces \[ \{X_\alpha\}_{\alpha\in(0,1]}. \] 
 		\item The auxiliary family \[ \{\widetilde{X}_\alpha\}_{\alpha\in(0,1]} \] captures the hierarchical structure of the $\alpha$-levels. Indeed, if $\alpha\ge\gamma$, then \[ [\alpha,1]\subseteq[\gamma,1], \] which immediately implies \[ \widetilde{X}_\alpha \subseteq \widetilde{X}_\gamma. \] Hence, \[ \{\widetilde{X}_\alpha\}_{\alpha\in(0,1]} \] forms a nested family of vector spaces, reflecting the behavior of $\alpha$-cuts in the theory of fuzzy sets.
 		 \item The proposed notion of a fuzzy vector is inspired by the parametric representation of fuzzy numbers. Instead of assigning a membership function directly to vectors, uncertainty is encoded through the family\[
 		 \tilde{x}:=\left\{x_\alpha\,:\,\alpha\in(0,1]\right\}.
 		 \] where the normalization condition \[ x_1=x \] ensures that every fuzzy vector possesses a unique underlying crisp representative. This viewpoint is consistent with the parametric approach developed for fuzzy numbers and provides a convenient framework for introducing fuzzy inner products and fuzzy norms. Moreover, it follows the same philosophy as the parametric representation of fuzzy numbers developed by Bede~\cite{BedeStefanini2006} and coauthors, where a fuzzy quantity is characterized by its family of $\alpha$-level representatives.

  \end{enumerate} \end{remark}

\begin{example} Let \[ V=\mathbb{R}^4, \qquad X=\{(a_1,a_2,0,0):a_1,a_2\in\mathbb{R}\}, \] which is a two-dimensional vector subspace of $\mathbb{R}^4$. For each \[ x=(a_1,a_2,0,0)\in X, \] define its $\alpha$-level representative by \[ x_\alpha = \bigl(a_1,\, a_2,\, p(\alpha)a_1,\, (1-\alpha)a_2 \bigr), \qquad \alpha\in(0,1], \] where \[ p(\alpha)= \begin{cases} 0, & \alpha\in\left(\frac12,1\right],\\[2mm] \alpha, & \alpha\in\left(0,\frac12\right]. \end{cases} \] Since \[ p(1)=0 \quad\text{and}\quad 1-1=0, \] we obtain \[ x_1=(a_1,a_2,0,0)=x. \] Hence, the normalization condition in Definition~\ref{def:fuzzyvector} is satisfied. Therefore, \[ \widetilde{x} = \left\{x_\alpha:\alpha\in(0,1]\right\} \] is a fuzzy vector associated with the crisp vector \[ x=(a_1,a_2,0,0). \] Moreover, for each $\alpha\in(0,1]$, the corresponding space \[ X_\alpha = \operatorname{span} \{x_\alpha:x\in X\} \] is given by \[ X_\alpha = \left\{ \bigl(a_1,\, a_2,\, p(\alpha)a_1,\, (1-\alpha)a_2 \bigr): a_1,a_2\in\mathbb{R} \right\}, \] which is a two-dimensional vector subspace of $\mathbb{R}^4$.
Moreover, \[ \widetilde{X}_\alpha = \operatorname{span} \{x_\beta:\;x\in X,\ \beta\in[\alpha,1]\}. \] If $\alpha>\frac12$, then $p(\beta)=0$ for every $\beta\in[\alpha,1]$, and hence \[ x_\beta = (a_1,a_2,0,(1-\beta)a_2). \] Therefore, \[ \widetilde{X}_\alpha = \operatorname{span} \{(1,0,0,0),(0,1,0,0),(0,0,0,1)\} = \{(u,v,0,w):u,v,w\in\mathbb{R}\}. \] On the other hand, if $0<\alpha\le\frac12$, then the interval $[\alpha,1]$ contains values $\beta\le\frac12$ for which $p(\beta)=\beta>0$. Consequently, the vectors \[ (1,0,\beta,0), \qquad (1,0,0,0), \] belong to the generating family, whose difference yields \[ (0,0,\beta,0), \] and hence $(0,0,1,0)$ also belongs to the span. It follows that \[ \widetilde{X}_\alpha = \mathbb{R}^4, \qquad 0<\alpha\le\frac12. \]
 \end{example}
 The following proposition shows that every fuzzy vector canonically induces a fuzzy set in the sense of Zadeh. Hence, the proposed framework extends the classical fuzzy-set representation while preserving its standard interpretation through $\alpha$-cuts. Moreover, unlike the classical approach, the present construction naturally gives rise to a family of vector spaces ${\widetilde{X}\alpha}_{\alpha\in(0,1]}$, providing an algebraic framework in which linear operations and further analytical structures can be developed.

\begin{remark}[Induced Fuzzy Set and Level Reconstruction]	\label{prop:Fuzzysetinducedfuzzyvectorxxxxx}
	Let
	\[
	\widetilde{x}=\{x_\alpha\}_{\alpha\in(0,1]}
	\]
	be a fuzzy vector whose parametrization
	\[
	\Phi:(0,1]\to V,\qquad
	\Phi(\alpha)=x_\alpha,
	\]
	is injective. Define the induced membership function
	\[
	\mu_{\widetilde{x}}:V\to[0,1]
	\]
	by
	\[
	\mu_{\widetilde{x}}(y)=
	\begin{cases}
		\alpha, & \text{if } y=x_\alpha \text{ for some }\alpha\in(0,1],\\
		0, & \text{otherwise}.
	\end{cases}
	\]
	Then each representative vector retains exactly its generating level,
	\[
	\mu_{\widetilde{x}}(x_\alpha)=\alpha,
	\qquad
	\forall\,\alpha\in(0,1].
	\]
	Consequently, every point $y\in\widetilde{x}$ can be recovered from its
	$\alpha$-cuts through
	\[
	\mu_{\widetilde{x}}(y)
	=
	\sup\{\alpha\in(0,1]:
	y\in[\mu_{\widetilde{x}}]_\alpha\},
	\]
	which is the standard level-reconstruction property of fuzzy sets.
	Moreover, since $\widetilde{x}$ contains only elements generated at positive
	levels, the induced membership values belong to $(0,1]$. By assigning the
	value $0$ to every vector in $V\setminus\widetilde{x}$, the function
	$\mu_{\widetilde{x}}$ becomes a conventional fuzzy subset of $V$.
\end{remark}

 \begin{lemma}[Injective Parametrization and Reconstruction]\label{lem:InjectiveRepresentation}
 Assume that the mapping
 	\[
 	\Phi:(0,1]\to V,
 	\qquad
 	\Phi(\alpha)=x_\alpha,
 	\]
 	is injective.
 	Then:
 	
 	\begin{enumerate}
 		
 		\item
 		For every \(\alpha\in(0,1]\),
 		\[
 		\mu_{\widetilde{x}}(x_\alpha)=\alpha.
 		\]
 		
 		\item
 		The membership function \(\mu_{\widetilde{x}}\) uniquely determines the
 		parameterization \(\alpha\mapsto x_\alpha\). More precisely, for every
 		\(\alpha\in(0,1]\), the element \(x_\alpha\) is the unique point of
 		\(\widetilde{x}\) satisfying
 		\[
 		\mu_{\widetilde{x}}(x_\alpha)=\alpha.
 		\]
 		Consequently, the fuzzy vector
 		\[
 		\widetilde{x}
 		=
 		\{x_\alpha\}_{\alpha\in(0,1]}
 		\]is uniquely determined by its induced membership function.
 		
 	\end{enumerate}
 	
 \end{lemma}
 
 \begin{proof}
 	Since \(\Phi\) is injective,
 	\[
 	x_\alpha=x_\beta
 	\iff
 	\alpha=\beta.
 	\]
 	Hence
 	\[
 	\{\beta\in(0,1]:x_\beta=x_\alpha\}
 	=
 	\{\alpha\},
 	\]
 	and therefore
 	\[
 	\mu_{\widetilde{x}}(x_\alpha)
 	=
 	\sup\{\beta:x_\beta=x_\alpha\}
 	=
 	\sup\{\alpha\}
 	=
 	\alpha,
 	\]
 	which proves (1).
 	Now let \(y\in\widetilde{x}\) satisfy
 	\[
 	\mu_{\widetilde{x}}(y)=\alpha.
 	\]
 	Since \(y\in\widetilde{x}\), there exists \(\beta\in(0,1]\) such that
 	\(y=x_\beta\).
 	Using (1),
 	\[
 	\alpha
 	=
 	\mu_{\widetilde{x}}(y)
 	=
 	\mu_{\widetilde{x}}(x_\beta)
 	=
 	\beta.
 	\]
 	Hence
 	\[
 	y=x_\beta=x_\alpha.
 	\]
 	Thus \(x_\alpha\) is uniquely determined by its membership value, and so
 	the entire family
 	\[
 	\widetilde{x}
 	=
 	\{x_\alpha\}_{\alpha\in(0,1]}
 	\]
 	is uniquely determined by \(\mu_{\widetilde{x}}\).
 	
 \end{proof}

\begin{proposition}[Reconstruction from Parametrized $\alpha$-Cuts]
	\label{prop:ReconstructionFromAlphaCuts}
	Let
	\[
	\widetilde{x}
	=
	\{x_\alpha\}_{\alpha\in(0,1]}
	\]
	be a fuzzy vector induced by an injective parametrization
	\[
	\Phi:(0,1]\to V,\qquad \Phi(\alpha)=x_\alpha.
	\]
	Assume that a fuzzy subset \(\mu:V\to[0,1]\) satisfies
	\[
	[\mu]_\alpha
	=
	\{x_\beta:\beta\ge\alpha\},
	\qquad \forall \alpha\in(0,1].
	\]Then
	\[
	\mu=\mu_{\widetilde{x}}.
	\]
	Consequently, the fuzzy vector \(\widetilde{x}\) is uniquely determined by its
	\(\alpha\)-cuts.
	
\end{proposition}
 
\begin{proof}
	Let \(y\in V\).
	If \(y\notin \{x_\alpha:\alpha\in(0,1]\}\), then
	by the assumption on the \(\alpha\)-cuts we have
	\(y\notin[\mu]_\alpha\) for every \(\alpha\in(0,1]\), hence
	\[
	\mu(y)=0=\mu_{\widetilde{x}}(y).
	\]Now assume that \(y=x_\gamma\) for some \(\gamma\in(0,1]\).
	By the assumed structure of the \(\alpha\)-cuts,
	\[
	y\in[\mu]_\alpha \iff \gamma\ge \alpha.
	\]
	Therefore,
	\[
	\mu(y)
	=
	\sup\{\alpha\in(0,1]: y\in[\mu]_\alpha\}
	=
	\sup\{\alpha\le\gamma\}
	=
	\gamma.
	\]On the other hand, injectivity implies
	\[
	\{\beta\in(0,1]: x_\beta=y\}=\{\gamma\},
	\]
	and hence
	\[
	\mu_{\widetilde{x}}(y)
	=
	\sup\{\beta:x_\beta=y\}
	=
	\gamma.
	\]Thus,
	\[
	\mu(y)=\mu_{\widetilde{x}}(y)
	\quad \text{for all } y\in V,
	\]
	which proves that \(\mu=\mu_{\widetilde{x}}\).
	
\end{proof}

 \begin{corollary}[α-cut representation] As a direct consequence of the previous lemma and proposition, the fuzzy set $\mu_{\tilde{x}}$ admits the classical $\alpha$-cut representation: \[ \mu_{\tilde{x}}(y) = \sup_{\alpha\in(0,1]} \alpha \,\chi_{(\mu_{\tilde{x}})_\alpha}(y), \qquad y\in V, \] where $\chi_{(\mu_{\tilde{x}})_\alpha}$ denotes the characteristic function of the $\alpha$-cut $(\mu_{\tilde{x}})_\alpha$. \end{corollary} \begin{proof} For each $y\in V$, we have \[ \mu_{\tilde{x}}(y) = \sup\{\alpha\in(0,1]: y\in [\mu_{\tilde{x}}]_\alpha\}. \] Using the identity \[ \chi_{[\mu_{\tilde{x}}]_\alpha}(y) = \begin{cases} 1, & y\in [\mu_{\tilde{x}}]_\alpha,\\ 0, & \text{otherwise}, \end{cases} \] we obtain \[ \alpha \,\chi_{[\mu_{\tilde{x}}]_\alpha}(y) = \begin{cases} \alpha, & y\in [\mu_{\tilde{x}}]_\alpha,\\ 0, & \text{otherwise}. \end{cases} \] Taking supremum over $\alpha\in(0,1]$ yields the desired representation.\qedhere\end{proof} which is consistent with the classical representation theorem of fuzzy sets via level sets. 
\begin{remark}
	Under the injectivity assumption, Lemma~\ref{lem:InjectiveRepresentation}
	shows that
	\[
	\mu_{\tilde{x}}(x_\alpha)=\alpha.
	\]
	Consequently, the membership function provides an order-preserving identification
	between the parameter set $(0,1]$ and the support
	\[
	\widetilde{x}=\{x_\alpha:\alpha\in(0,1]\}.
	\]
\end{remark}

 \begin{remark}
 	Under the injectivity assumption of
 	Lemma~\ref{lem:InjectiveRepresentation},
 	the underlying crisp vector is uniquely characterized by its membership
 	value. Indeed,
 	\[
 	\mu_{\widetilde{x}}(x)
 	=
 	\mu_{\widetilde{x}}(x_1)
 	=
 	1,
 	\]
 	and if \(y\in\widetilde{x}\) satisfies
 	\[
 	\mu_{\widetilde{x}}(y)=1,
 	\]
 	then \(y=x_1=x\). Consequently,
 	\[
 	\mu_{\widetilde{x}}(y)<1,
 	\qquad
 	\forall\,y\in\widetilde{x}\setminus\{x\}.
 	\]
 	Thus, the underlying crisp vector is the unique element of the fuzzy
 	vector having maximal membership.
 \end{remark}

 \begin{remark} The above proposition provides a mathematical justification for viewing the induced membership function as an interface between the crisp and fuzzy descriptions of a vector. The normalization condition \[ x_1=x \] embeds the underlying crisp vector into the fuzzy framework, while the identity \[ \mu_{\tilde{x}}(x_\alpha)=\alpha \] shows that each level representative is assigned precisely its corresponding confidence level. Consequently, the family \[ \{x_\alpha\}_{\alpha\in(0,1]} \] describes a gradual transition from the crisp vector to its fuzzy realization, and the induced membership function preserves this parametric information within the classical fuzzy-set framework. \end{remark}
 
 \begin{example}
	Let $V=\mathbb{R}^2$ and define a fuzzy vector $\tilde{x}=\{x_\alpha\}_{\alpha\in(0,1]}$ by
	\[
	x_\alpha = (\alpha, 1-\alpha), \qquad \alpha\in(0,1].
	\]
	Then the induced fuzzy set $\mu_{\tilde{x}}$ is given by
	\[
	\mu_{\tilde{x}}(x_\alpha)=\alpha.
	\]
	Hence, the $\alpha$-cuts are
	\[
	[\mu_{\tilde{x}}]_\alpha
	=
	\{x_\beta:\beta\ge \alpha\}
	=
	\{(\beta,1-\beta): \beta\ge \alpha\}.
	\]
		In particular, we have:
	\[
	[\mu_{\tilde{x}}]_\alpha
	=
	\{(t,1-t): t\in[\alpha,1]\}.
	\]
	Now fix a point $y=(y_1,y_2)\in \mathbb{R}^2$. Then
	\[
	\mu_{\tilde{x}}(y)
	=
	\sup\{\alpha\in(0,1]: y=(\alpha,1-\alpha)\}.
	\]
	Therefore:
	\[
	\mu_{\tilde{x}}(y)=
	\begin{cases}
		\alpha, & \text{if } y=(\alpha,1-\alpha)\text{ for some }\alpha\in(0,1],\\[2mm]
		0, & \text{otherwise}.
	\end{cases}
	\]
	Finally, using the α-cut representation theorem, we obtain:
	\[
	\mu_{\tilde{x}}(y)
	=
	\sup_{\alpha\in(0,1]} \alpha\,\chi_{[\mu_{\tilde{x}}]_\alpha}(y),
	\]
	which explicitly reconstructs the membership function from its α-cuts.
\end{example}

\begin{theorem}[Change of Level Parameter Theorem]
	Let
	\[
	\widetilde{x}=\{x_\alpha\}_{\alpha\in(0,1]}
	\]
	be a fuzzy vector satisfying the assumptions of Proposition~\ref{prop:ReconstructionFromAlphaCuts}. Let
	\[
	f:(0,1]\longrightarrow I
	\]
	be a strictly increasing bijection onto an interval
	\(I\subseteq \mathbb{R}\).
	Define a new parameter
	\[
	p=f(\alpha),
	\]
	and let
	\[
	\alpha=f^{-1}(p).
	\]Then:
	
	\begin{enumerate}
		\item For every $\alpha\in(0,1]$,
		\[
		p=f\bigl(\mu_{\widetilde{x}}(x_\alpha)\bigr).
		\]
		
		\item The family $\widetilde{x}$ admits the equivalent representation
		\[
		\widetilde{x}
		=
		\{x_{f^{-1}(p)}\}_{p\in I}.
		\]
		\item For every $p\in I$,
		\[
		\mu_{\widetilde{x}}
		\bigl(x_{f^{-1}(p)}\bigr)
		=
		f^{-1}(p).
		\]
		
		\item For every $p_1,p_2\in I$,
		\[
		p_1<p_2
		\quad\Longleftrightarrow\quad
		f^{-1}(p_1)<f^{-1}(p_2).
		\]
		Hence the ordering of fuzzy levels is preserved under the
		reparametrization.
	\end{enumerate}
\end{theorem}

\begin{proof}
	Since the parametrization
	\[
	\Phi:(0,1]\to V,
	\qquad
	\Phi(\alpha)=x_\alpha,
	\]
	is injective, Lemma~\ref{lem:InjectiveRepresentation} yields
	\[
	\mu_{\widetilde{x}}(x_\alpha)=\alpha,
	\qquad
	\forall \alpha\in(0,1].
	\]
	Applying $f$ to both sides gives
	\[
	f\bigl(\mu_{\widetilde{x}}(x_\alpha)\bigr)
	=
	f(\alpha)
	=
	p,
	\]
	which proves (1).
	Since $f$ is bijective, every $p\in I$ corresponds to a unique level
	\[
	\alpha=f^{-1}(p).
	\]
	Therefore,
	\[
	\{x_\alpha\}_{\alpha\in(0,1]}
	=
	\{x_{f^{-1}(p)}\}_{p\in I},
	\]
	which proves (2).
	Using again Proposition~3.2, we obtain
	\[
	\mu_{\widetilde{x}}
	\bigl(x_{f^{-1}(p)}\bigr)
	=
	f^{-1}(p),
	\]
	proving (3).
	Since $f$ is strictly increasing, its inverse $f^{-1}$ is also strictly
	increasing. Hence
	\[
	p_1<p_2
	\quad\Longleftrightarrow\quad
	f^{-1}(p_1)<f^{-1}(p_2),
	\]
	which proves (4).
\end{proof}

\begin{example}
	Consider the vector space
	\[
	V=\mathbb{R}^{2},
	\]
	and let
	\[
	x=(1,0)\in V.
	\]
	Define a fuzzy vector
	\[
	\widetilde{x}
	=
	\{x_\alpha\}_{\alpha\in(0,1]}
	\]
	by
	\[
	x_\alpha=(1,1-\alpha),
	\qquad
	\alpha\in(0,1].
	\]
	Since
	\[
	x_1=(1,0)=x,
	\]
	the normalization condition is satisfied.
	Moreover, the parametrization
	\[
	\Phi:(0,1]\to\mathbb{R}^2,
	\qquad
	\Phi(\alpha)=x_\alpha,
	\]
	is injective. Indeed, if
	\[
	x_{\alpha_1}=x_{\alpha_2},
	\]
	then
	\[
	(1,1-\alpha_1)=(1,1-\alpha_2),
	\]
	which implies
	\[
	\alpha_1=\alpha_2.
	\]Hence, by Proposition~3.2,
	\[
	\mu_{\widetilde{x}}(x_\alpha)=\alpha,
	\qquad
	\forall \alpha\in(0,1].
	\]
	Now consider the strictly increasing bijection
	\[
	f:(0,1]\to [0,\infty),
	\qquad
	f(\alpha)=-\ln(\alpha).
	\]
	Let
	\[
	p=f(\alpha).
	\]
	Then
	\[
	\alpha=e^{-p}.
	\]
	By the Change of Level Parameter Theorem,
	\[
	\widetilde{x}
	=
	\{x_{e^{-p}}\}_{p\ge0}.
	\]Substituting \(\alpha=e^{-p}\) into the definition of \(x_\alpha\), we obtain
	\[
	x_{e^{-p}}
	=
	(1,1-e^{-p}),
	\qquad p\ge0.
	\]Furthermore,
	\[
	\mu_{\widetilde{x}}(x_{e^{-p}})
	=
	e^{-p}.
	\]For example,
	\[
	p=0
	\quad\Longrightarrow\quad
	x_{e^{-p}}=(1,0),
	\]
	and
	\[
	\mu_{\widetilde{x}}(x_{e^{-p}})=1.
	\]Also,
	\[
	p=\ln 2
	\quad\Longrightarrow\quad
	x_{e^{-p}}
	=
	\left(1,\frac12\right),
	\]
	and
	\[
	\mu_{\widetilde{x}}
	\left(
	1,\frac12
	\right)
	=
	\frac12.
	\]
	Therefore, the parameter \(p\) provides an equivalent representation of
	the same fuzzy vector, replacing the level parameter \(\alpha\in(0,1]\)
	with the additive parameter \(p\in[0,\infty)\).
\end{example}

\section{Structural Analysis of Fuzzy Spaces Induced by Fuzzy Vectors}
\begin{proposition}
[Dimension Profile]
	Assume that $X$ is finite-dimensional and let
	\[
	\widetilde X_\alpha
	=
	\operatorname{span}
	\{x_\beta:\beta\in[\alpha,1],\,x\in X\}.
	\]
	
	Define
	\[
	\delta:(0,1]\longrightarrow\mathbb N,
	\qquad
	\delta(\alpha)=\dim(\widetilde X_\alpha).
	\]
	Then:
	
	\begin{enumerate}
		\item
		If
		\[
		\alpha_1<\alpha_2,
		\]
		then
		\[
		\widetilde X_{\alpha_2}
		\subseteq
		\widetilde X_{\alpha_1},
		\]
		and consequently
		\[
		\delta(\alpha_2)
		\le
		\delta(\alpha_1).
		\]
		
		\item
		The function $\delta$ is non-increasing.
		
		\item
		Since
		\[
		\dim(X)<\infty,
		\]
		the function $\delta$ assumes only finitely many values.
		Hence there exist numbers
		\[
		1=\alpha_0>\alpha_1>\cdots>\alpha_m>0
		\]
		such that $\delta$ is constant on each interval
		\[
		(\alpha_i,\alpha_{i-1}],
		\]
		and decreases only at the finitely many critical levels
		\[
		\alpha_1,\ldots,\alpha_m.
		\]
	\end{enumerate}
	Therefore the family
	\[
	\{\widetilde X_\alpha\}_{\alpha\in(0,1]}
	\]
	admits a finite stratification according to dimension.
\end{proposition}
\begin{proof}
	If $\alpha_1<\alpha_2$, then
	\[
	[\alpha_2,1]\subseteq[\alpha_1,1].
	\]
	Therefore
	\[
	\widetilde X_{\alpha_2}
	=
	\operatorname{span}
	\{x_\beta:\beta\ge\alpha_2\}
	\subseteq
	\operatorname{span}
	\{x_\beta:\beta\ge\alpha_1\}
	=
	\widetilde X_{\alpha_1}.
	\]
	Hence
	\[
	\dim(\widetilde X_{\alpha_2})
	\le
	\dim(\widetilde X_{\alpha_1}),
	\]
	proving that $\delta$ is non-increasing.
	Since
	\[
	\delta(\alpha)\in
	\{0,1,\ldots,\dim(V)\},
	\]
	it can change value only finitely many times.
	Consequently, $\delta$ is constant on finitely many intervals and decreases only at finitely many critical levels.
\end{proof}
\begin{proposition}[Invariance of Tail Spaces under Monotone Reparameterization]
	Let
	\[
	\widetilde{x}
	=
	\{x_\alpha\}_{\alpha\in(0,1]}
	\]
	be a fuzzy vector, and let
	\[
	g:(0,1]\longrightarrow(0,1]
	\]
	be a strictly increasing bijection.
	Define a reparameterized fuzzy vector
	\[
	\widetilde{y}
	=
	\{y_t\}_{t\in(0,1]}
	\]
	by
	\[
	y_t=x_{g(t)},
	\qquad t\in(0,1].
	\]
	For every $\alpha\in(0,1]$, define
	\[
	\widetilde{Y}_\alpha
	=
	\operatorname{span}
	\{y_t:\;t\in[\alpha,1]\}.
	\]
	Then
	\[
	\widetilde{Y}_\alpha
	=
	\widetilde{X}_{g(\alpha)},
	\]
	where
	\[
	\widetilde{X}_{g(\alpha)}
	=
	\operatorname{span}
	\{x_\beta:\beta\in[g(\alpha),1]\}.
	\]
	Consequently,
	\[
	\{\widetilde{Y}_\alpha:\alpha\in(0,1]\}
	=
	\{\widetilde{X}_\alpha:\alpha\in(0,1]\},
	\]
	that is, the family of tail spaces is invariant under every monotone
	reparameterization of the level parameter.
	Consequently, every structural property depending only on the nested family
	\[
	\{\widetilde X_\alpha\}_{\alpha\in(0,1]}
	\]
	is invariant under monotone reparameterization.
	In particular, the dimension profile
	\[
	\delta(\alpha)
	=
	\dim(\widetilde X_\alpha)
	\]
	is preserved up to composition with the change of parameter,
	\[
	\delta_Y(\alpha)
	=
	\delta_X(g(\alpha)).
	\]
	Hence the set of critical levels at which the dimension changes is an
	order-theoretic invariant.
\end{proposition}

\begin{proof}
	By definition,
	\[
	\widetilde{Y}_\alpha
	=
	\operatorname{span}
	\{y_t:t\in[\alpha,1]\}
	=
	\operatorname{span}
	\{x_{g(t)}:t\in[\alpha,1]\}.
	\]
	Since $g$ is a strictly increasing bijection,
	\[
	t\in[\alpha,1]
	\quad\Longleftrightarrow\quad
	g(t)\in[g(\alpha),1].
	\]
	Therefore,
	\[
	\{x_{g(t)}:t\in[\alpha,1]\}
	=
	\{x_\beta:\beta\in[g(\alpha),1]\},
	\]
	where we have written $\beta=g(t)$.
	Taking spans yields
	\[
	\widetilde{Y}_\alpha
	=
	\operatorname{span}
	\{x_\beta:\beta\in[g(\alpha),1]\}
	=
	\widetilde{X}_{g(\alpha)}.
	\]
	Finally, since $g$ is bijective,
	\[
	\{g(\alpha):\alpha\in(0,1]\}
	=
	(0,1],
	\]
	and hence
	\[
	\{\widetilde{Y}_\alpha:\alpha\in(0,1]\}
	=
	\{\widetilde{X}_{g(\alpha)}:\alpha\in(0,1]\}
	=
	\{\widetilde{X}_\alpha:\alpha\in(0,1]\}.
	\]
	Thus the collection of tail spaces depends only on the ordering of the
	levels and not on the particular numerical parametrization.
\end{proof}
\begin{corollary}[Invariance of the Dimension Profile]
	Let
	\[
	\delta_X(\alpha)
	=
	\dim(\widetilde X_\alpha),
	\qquad
	\delta_Y(\alpha)
	=
	\dim(\widetilde Y_\alpha).
	\]Then
	\[
	\delta_Y(\alpha)
	=
	\delta_X(g(\alpha)),
	\qquad
	\forall\alpha\in(0,1].
	\]Consequently,
	
	\begin{enumerate}
		\item
		The family of attained dimensions is preserved.
		
		\item
		The jump points of the dimension function are mapped by
		\[
		\alpha
		\longmapsto
		g^{-1}(\alpha).
		\]
		
		\item
		The ordered sequence of dimensions
		\[
		\dim(\widetilde X_1),
		\dim(\widetilde X_{\alpha}),
		\ldots
		\]
		is preserved up to the monotone change of parameter.
		Hence the dimension profile of a fuzzy labeling is an
		order-theoretic invariant.
	\end{enumerate}
\end{corollary}

\begin{proof}
	By the previous theorem,
	\[
	\widetilde Y_\alpha
	=
	\widetilde X_{g(\alpha)}.
	\]
	Taking dimensions immediately gives
	\[
	\delta_Y(\alpha)
	=
	\dim(\widetilde Y_\alpha)
	=
	\dim(\widetilde X_{g(\alpha)})
	=
	\delta_X(g(\alpha)).
	\]The remaining assertions follow from the fact that
	\(g\) is an order-preserving bijection.
\end{proof}

\begin{definition}[Induced Membership Function]\label{def:induced-membership}
	Let
	\[
	\{\widetilde{X}_{\alpha}\}_{\alpha\in(0,1]}
	\]
	be the nested family of tail spaces associated with a fuzzy labeling.
	Define the mapping
	\[
	\nu:V\longrightarrow [0,1]
	\]
	by
	\[
	\nu(v)=
	\sup\Bigl\{
	\alpha\in(0,1]:
	v\in\widetilde{X}_{\alpha}
	\Bigr\},
	\]
	where the supremum of the empty set is taken to be \(0\).
	The value \(\nu(v)\) is called the persistence level of the vector \(v\).
\end{definition}
\begin{proposition}\label{prop:scalar-invariance-membership}
	Let
	\[
	\nu(v)=\sup\{\alpha\in(0,1]:v\in\widetilde{X}_{\alpha}\}
	\]
	be the membership function induced by the nested family
	$\{\widetilde{X}_{\alpha}\}_{\alpha\in(0,1]}$.
	Then, for every $x\in\widetilde{X}$ and every nonzero scalar
	$a\in\mathbb{F}$,
	\[
	\nu(ax)=\nu(x).
	\]
	Moreover,
	\[
	\nu(0)=1.
	\]
\end{proposition}

\begin{proof}
	Since each $\widetilde{X}_{\alpha}$ is a vector subspace, for every
	$\alpha\in(0,1]$ and every nonzero scalar $a$,
	\[
	x\in\widetilde{X}_{\alpha}
	\quad\Longleftrightarrow\quad
	ax\in\widetilde{X}_{\alpha}.
	\]
	Hence,
	\[
	\{\alpha:x\in\widetilde{X}_{\alpha}\}
	=
	\{\alpha:ax\in\widetilde{X}_{\alpha}\},
	\]
	and therefore
	\[
	\nu(ax)
	=
	\sup\{\alpha:ax\in\widetilde{X}_{\alpha}\}
	=
	\sup\{\alpha:x\in\widetilde{X}_{\alpha}\}
	=
	\nu(x).
	\]
	
	Finally, since $0$ belongs to every vector subspace
	$\widetilde{X}_{\alpha}$, we have
	\[
	0\in\widetilde{X}_{\alpha},
	\qquad \forall\,\alpha\in(0,1],
	\]
	which implies
	\[
	\nu(0)
	=
	\sup(0,1]
	=
	1.
	\]
	This completes the proof.
\end{proof}
\begin{theorem}[Reparameterization Invariance of the Dimension Profile]
	\label{thm:DimensionProfileInvariant}
	
	Let
	\[
	\{\widetilde X_\alpha\}_{\alpha\in(0,1]}
	\quad\text{and}\quad
	\{\widetilde Y_\alpha\}_{\alpha\in(0,1]}
	\]
	be two nested families of vector spaces arising from parameterized vector
	structures.
	Assume that there exists a strictly increasing bijection
	\[
	g:(0,1]\longrightarrow(0,1]
	\]
	such that
	\[
	\widetilde Y_\alpha
	=
	\widetilde X_{g(\alpha)},
	\qquad
	\forall\alpha\in(0,1].
	\]Let
	\[
	\delta_X(\alpha)=\dim(\widetilde X_\alpha),
	\qquad
	\delta_Y(\alpha)=\dim(\widetilde Y_\alpha)
	\]
	denote the corresponding dimension profiles.\\Then
	\[
	\delta_Y
	=
	\delta_X\circ g,
	\]
	that is,
	\[
	\delta_Y(\alpha)
	=
	\delta_X(g(\alpha)),
	\qquad
	\forall\alpha\in(0,1].
	\]Consequently, if there exists no strictly increasing bijection
	\[
	g:(0,1]\longrightarrow(0,1]
	\]
	such that
	\[
	\delta_Y=\delta_X\circ g,
	\]
	then the two parameterized vector structures cannot be equivalent under any
	strictly increasing reparameterization.
	
\end{theorem}

\begin{proof}
	
	By assumption,
	\[
	\widetilde Y_\alpha
	=
	\widetilde X_{g(\alpha)},
	\qquad
	\forall\alpha\in(0,1].
	\]
	Taking dimensions of both sides gives
	\[
	\dim(\widetilde Y_\alpha)
	=
	\dim(\widetilde X_{g(\alpha)}).
	\]
	Since
	\[
	\delta_X(\alpha)
	=
	\dim(\widetilde X_\alpha),
	\qquad
	\delta_Y(\alpha)
	=
	\dim(\widetilde Y_\alpha),
	\]
	we obtain
	\[
	\delta_Y(\alpha)
	=
	\delta_X(g(\alpha)),
	\qquad
	\forall\alpha\in(0,1].
	\]
	Hence
	\[
	\delta_Y=\delta_X\circ g,
	\]
	showing that the dimension profile is invariant under every strictly
	increasing reparameterization.
	For the second assertion, suppose that no strictly increasing bijection
	\[
	g:(0,1]\to(0,1]
	\]
	satisfies\[
	\delta_Y=\delta_X\circ g.
	\]
	If the two parameterized vector structures were equivalent under a strictly
	increasing reparameterization, then the first part of the theorem would imply
	\[
	\delta_Y=\delta_X\circ g,
	\]
	contradicting the assumption.
	Therefore such a reparameterization cannot exist.
	
\end{proof}

\begin{remark}
	The previous theorem shows that the dimension profile is an intrinsic
	structural invariant of parameterized vector structures under strictly
	increasing reparameterizations.
	Consequently, it provides a practical criterion for distinguishing
	parameterized vector structures without explicitly constructing a
	reparameterization.
	Moreover, unlike the classical dimension of the ambient vector space,
	the dimension profile records how the associated nested family of vector
	spaces evolves with the parameter.
	Hence it contains structural information that is invisible to the ambient
	dimension alone.
	
\end{remark}

\begin{remark}
	In general, the dimension profile is not expected to be a complete
	invariant.
	Different nested families of vector spaces may induce the same dimension
	profile while having different algebraic configurations.
	Therefore equality of dimension profiles (up to a strictly increasing
	reparameterization) is a necessary but generally not sufficient condition
	for the equivalence of parameterized vector structures.
	
\end{remark}

\begin{example} Let \[ V=\mathbb{R}^4 \] with the standard basis \[ e_1,e_2,e_3,e_4. \] Assume that the associated tail spaces are \[ \widetilde X_\alpha= \begin{cases} \operatorname{span}\{e_1,e_2\}, & \frac12<\alpha\le1, \\[2mm] \operatorname{span}\{e_1,e_2,e_3\}, & \frac14<\alpha\le\frac12, \\[2mm] \operatorname{span}\{e_1,e_2,e_3,e_4\}, & 0<\alpha\le\frac14. \end{cases} \] Hence \[ \delta(\alpha)= \begin{cases} 2, & \frac12<\alpha\le1, \\[2mm] 3, & \frac14<\alpha\le\frac12, \\[2mm] 4, & 0<\alpha\le\frac14. \end{cases} \] Thus the dimension profile has two jumps, corresponding to the appearance of the new directions \(e_3\) and \(e_4\). \end{example}

\begin{figure}[ht]
	\centering
	
	\begin{tikzpicture}[node distance=1.6cm]
		
		\tikzset{
			box/.style={
				draw,
				rounded corners,
				minimum width=6cm,
				minimum height=9mm,
				align=center
			}
		}
		
		\node[box] (A)
		{$\widetilde X_{1}=\operatorname{span}\{e_1,e_2\}$};
		
		\node[box,below of=A] (B)
		{$\widetilde X_{\frac12}=\operatorname{span}\{e_1,e_2,e_3\}$};
		
		\node[box,below of=B] (C)
		{$\widetilde X_{\frac14}=\operatorname{span}\{e_1,e_2,e_3,e_4\}=V$};
		
		\draw[->,thick] (A)--(B);
		\draw[->,thick] (B)--(C);
		
		\node at (4.2,0) {$\dim=2$};
		\node at (4.2,-1.6) {$\dim=3$};
		\node at (4.2,-3.2) {$\dim=4$};
		
	\end{tikzpicture}
	
	\caption{Nested family of tail spaces and their dimensions.}
	
\end{figure}
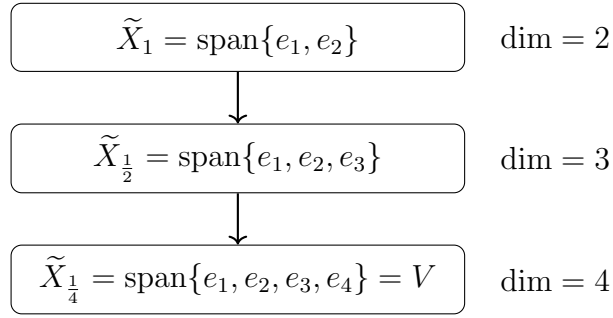

\section{Representation Theorems for Lubczonok Fuzzy Vector Spaces}

\begin{theorem}[Canonical Induction of a Lubczonok Fuzzy Vector Space]
	\label{thm:canonical-induction}
	Let
	\[
	\mu:(0,1]\times X\longrightarrow V
	\]
	be a fuzzy labeling, and let
	\[
	\{\widetilde X_\alpha\}_{\alpha\in(0,1]}
	\]
	be the associated nested family of tail spaces.Define
		\[
	X_0:=
	\bigcup_{\alpha\in(0,1]}
	\widetilde X_\alpha
	=
	V,
	\]
	and
	\[
	\nu:V\longrightarrow[0,1]
	\]
	by
	\[
	\nu(v)
	=
	\sup\{\alpha\in(0,1]:v\in\widetilde X_\alpha\},
	\]
	where the supremum of the empty set is taken to be \(0\).
	Then the pair
	\[
	(V,\nu)
	\]
	is a fuzzy vector space in the sense of Lubczonok.
\end{theorem}

\begin{proof}
	
Let
	\[
	x,y\in V,
	\qquad
	a,b\in\mathbb F,
	\]
	and set
	\[
	\gamma
	=
	\nu(x)\wedge\nu(y).
	\]
	If
	\[
	\gamma=0,
	\]
	then
	\[
	\nu(ax+by)\ge0=\gamma,
	\]
	and the desired inequality is immediate.
	Assume now that
	\[
	\gamma>0.
	\]
	Let
	\[
	0<\varepsilon<\gamma.
	\]
	By the definition of \(\nu\),
	\[
	x,y\in
	\widetilde X_{\gamma-\varepsilon}.
	\]
	Since
	\(
	\widetilde X_{\gamma-\varepsilon}
	\)
	is a vector subspace of \(V\),
	\[
	ax+by
	\in
	\widetilde X_{\gamma-\varepsilon}.
	\]
	Therefore,
	\[
	\nu(ax+by)
	\ge
	\gamma-\varepsilon.
	\]
	Since \(\varepsilon>0\) is arbitrary,
	\[
	\nu(ax+by)
	\ge
	\gamma
	=
	\nu(x)\wedge\nu(y).
	\]
	Hence, for every
	\(a,b\in\mathbb F\) and every
	\(x,y\in V\),
	\[
	\nu(ax+by)
	\ge
	\nu(x)\wedge\nu(y),
	\]
	which is precisely the defining axiom of a fuzzy vector space in the sense
	of Lubczonok.
	Therefore,
	\[
	(V,\nu)
	\]
	is a Lubczonok fuzzy vector space.
	
\end{proof}

	\begin{remark}
		Consider the vector space $V=\mathbb{R}^{2}$ equipped with the labeling
		\[
		\mu:(0,1]\times V\longrightarrow V,
		\qquad
		\mu\bigl(\alpha,(x_{1},x_{2})\bigr)
		=
		(\alpha x_{1},\alpha^{2}x_{2}).
		\]
		If we start from the vector
		\[
		x=(1,0),
		\]
		then
		\[
		\widetilde{x}
		=
		\{(\alpha,0):\alpha\in(0,1]\},
		\]
		and hence
		\[
		\operatorname{span}(\widetilde{x})
		=
		\operatorname{span}\{(1,0)\}
		=
		\{(a,0):a\in\mathbb{R}\},
		\]
		which is a one-dimensional subspace of $V$.
		On the other hand, if we start from the vector
		\[
		y=(1,1),
		\]
		then
		\[
		\widetilde{y}
		=
		\{(\alpha,\alpha^{2}):\alpha\in(0,1]\}.
		\]
		For instance, the vectors
		\[
		(1,1)
		\quad\text{and}\quad
		\left(\frac12,\frac14\right)
		\]
		are linearly independent. Therefore,
		\[
		\operatorname{span}(\widetilde{y})
		=
		\mathbb{R}^{2}.
		\]
		This example shows that the dimension of the Lubczonok fuzzy vector space generated by a labeled vector depends not only on the labeling function but also on the initial vector.
	\end{remark}

	\begin{theorem}[Representation of a Lubczonok Fuzzy Vector Space]
		\label{thm:representation}
		Let $(V,\nu)$ be a finite-dimensional Lubczonok fuzzy vector space.
		Assume that there exists a basis
		\[
		\mathcal B=\{e_1,\ldots,e_n\}
		\]
			such that
		\[
		1>\nu(e_1)>\nu(e_2)>\cdots>\nu(e_n)>0.
		\]
	   Choose a vector
		\[
		x=\sum_{i=1}^{n}c_ie_i,
		\qquad
		c_i\neq0.
		\]
		Define the labeling
		\[
		\mu:(0,1]\times\{x\}\longrightarrow V
		\]
		by
		\[
		\mu(1,x)=x,
		\]
		
		\[
		\mu(\nu(e_i),x)=e_i,
		\qquad
		i=1,\ldots,n,
		\]
		and
		\[
		\mu(\alpha,x)=x_{f(\alpha)},
		\]
		for
		\[
		\alpha\notin
		\{1,\nu(e_1),\ldots,\nu(e_n)\},
		\]
		where
		\[
		f:
		\left(0,\nu(e_n)\right)
		\longrightarrow
		\left(0,\nu(e_n)\right)
		\]
		is any bijection.
		Let
		\[
		\nu_\mu
		\]
		denote the Lubczonok membership function reconstructed from the above
		labeling.
		Then
		\[
		\nu_\mu=\nu.
		\]
		Hence every finite-dimensional Lubczonok fuzzy vector space admits a
		generating fuzzy labeling whose induced Lubczonok fuzzy vector space
		coincides with the original one.
		
	\end{theorem}
	
	\begin{proof}
		From\cite{Lubczonok1990}   we know that 	\[
		\nu\!\left(\sum_{i=1}^{n}c_ie_i\right)
		=
		\min\{\nu(e_i):c_i\neq0\}.
		\]
		Since
		\[
		\mu(1,x)=x,
		\]
		the original generating vector is preserved.
		Moreover, for every basis vector,
		\[
		\mu(\nu(e_i),x)=e_i,
		\]
		so each basis vector appears exactly at its prescribed Lubczonok level.
		The values of the labeling at all remaining levels are obtained through
		the bijection
		\[
		f,
		\]
		which only reparameterizes the labeling and therefore does not alter the
		corresponding Lubczonok membership values.
		Since every vector admits a unique expansion with respect to the basis
		\(\mathcal B\), and
		\[
		\nu\!\left(\sum_{i=1}^{n}c_ie_i\right)
		=
		\min\{\nu(e_i):c_i\neq0\},
		\]
		the reconstructed membership of every vector is determined by the
		minimum membership level of the basis vectors occurring in its
		representation. Consequently,
		\[
		\nu_\mu(y)=\nu(y),
		\qquad
		\forall\,y\in V.
		\]
		Hence the fuzzy labeling induces exactly the original Lubczonok fuzzy
		vector space.
		
	\end{proof}
\begin{example}
	Let
	\[
	V=\mathbb{R}^2
	\]
	with basis
	\[
	\{e_1,e_2\}.
	\]
	Define the Lubczonok membership function by
	\[
	\nu(e_1)=0.8,
	\qquad
	\nu(e_2)=0.4,
	\]
	and, for every nonzero vector
	\[
	y=ae_1+be_2,
	\]
	set
	\[
	\nu(y)=
	\min\{\nu(e_i):c_i\neq0\},
	\]
	where
	\[
	y=\sum c_ie_i.
	\]
	Choose the generating vector
	\[
	x=e_1+e_2.
	\]
	Define
	\[
	\mu_i(1,x)=x,
	\qquad
	\mu_i(0.8,x)=e_1,
	\qquad
	\mu_i(0.4,x)=e_2,
	\]
	for \(i=1,2\).
	Now let
	\[
	f_1,f_2:(0,0.4)\longrightarrow(0,0.4)
	\]
	be two different bijections, for example,
	\[
	f_1(t)=t,
	\]
	and
	\[
	f_2(t)=\frac{t^2}{0.4}.
	\]
	For every
	\[
	0<\alpha<0.4,
	\]
	define
	\[
	\mu_1(\alpha,x)=x_{f_1(\alpha)},
	\]
	and
	\[
	\mu_2(\alpha,x)=x_{f_2(\alpha)}.
	\]
	Since
	
	\[
	f_1\neq f_2,
	\]
	the two labelings are different,
	\[
	\mu_1\neq\mu_2.
	\]
	However, both labelings assign exactly the same Lubczonok levels to the
	basis vectors,
	\[
	\nu(e_1)=0.8,
	\qquad
	\nu(e_2)=0.4,
	\]
	and therefore, by the representation theorem, both induce the same
	Lubczonok fuzzy vector space.
	Hence a single Lubczonok fuzzy vector space admits infinitely many
	distinct fuzzy labelings.
	
\end{example}
\begin{remark} 	The representation theorem has several important consequences. First, 	every finite-dimensional Lubczonok fuzzy vector space admits a 	constructive parameterization by means of a fuzzy labeling. Starting only 	from the Lubczonok membership function, one can explicitly construct a 	generating fuzzy labeling whose representatives generate the entire 	vector space. In fact, by construction, all basis vectors appear as 	labels of a single generating vector. Since these basis vectors span 	\(V\), the whole vector space together with its Lubczonok fuzzy 	structure can be reconstructed from the labels of that single vector. 	Thus, the proposed framework provides not merely a representation of 	Lubczonok fuzzy vector spaces but also an explicit reconstruction 	mechanism for their underlying algebraic structure. \end{remark}
 \begin{remark} 	The representation theorem applies precisely because the family of 	subspaces generated by a fuzzy labeling is always a filtration. Indeed, 	the associated tail spaces satisfy 	\[ 	\alpha_1<\alpha_2 	\quad\Longrightarrow\quad 	\widetilde{X}_{\alpha_2}\subseteq 	\widetilde{X}_{\alpha_1}. 	\] 	Similarly, for a Lubczonok fuzzy vector space, each level set 	\[ 	V_\alpha=\{x\in V:\nu(x)\ge\alpha\} 	\] 	is a vector subspace satisfying 	\[ 	\alpha_1<\alpha_2 	\quad\Longrightarrow\quad 	V_{\alpha_2}\subseteq V_{\alpha_1}. 	\] 	Hence, both frameworks are naturally governed by the same filtration 	structure, making the representation theorem possible. 	Moreover, if the labeling map 	\[ 	L_x:(0,1]\longrightarrow V, 	\qquad 	L_x(\alpha)=x_\alpha, 	\] 	is injective, then the parameter induces the canonical ordering 	\[ 	x_\alpha\preceq x_\beta 	\quad\Longleftrightarrow\quad 	\alpha\le\beta. 	\] 	Consequently, every representative possesses a unique position in the 	labeling process. This ordered information is intrinsic to the labeling 	framework and disappears after passing to the induced Lubczonok 	membership function, which records only membership values.
  \end{remark}

\section{Conclusion and Future Research Directions}
In this paper, we introduced a new framework for fuzzy vectors based on fuzzy labelings and their associated families of nested tail spaces. Unlike the classical pointwise description of fuzzy vectors, the proposed approach represents each vector by an $\alpha$-indexed family of representatives, allowing the algebraic structure of fuzzy vectors to be studied through the geometry of the corresponding subspaces. This viewpoint led naturally to the construction of induced membership functions, dimension profiles, and other structural invariants associated with fuzzy labelings.
Based on this framework, we established a close relationship between fuzzy labelings and Lubczonok fuzzy vector spaces.\\ We proved that every fuzzy labeling canonically induces a Lubczonok fuzzy vector space through its family of tail spaces. Conversely, under natural finite-dimensional assumptions, every Lubczonok fuzzy vector space admits a representing fuzzy labeling. These representation results show that Lubczonok fuzzy vector spaces provide an abstract description of the structural information encoded by fuzzy labelings. At the same time, the non-uniqueness of the representation demonstrates that a Lubczonok fuzzy vector space determines, in general, an equivalence class of fuzzy labelings rather than a unique one. Consequently, fuzzy labelings retain additional geometric information that is not completely captured by the associated membership function.\\
Another contribution of this work is the introduction of the dimension profile and the associated dimension signature as new structural invariants of fuzzy labelings. These invariants describe how the dimensions of the nested tail spaces evolve with the level parameter and provide finer algebraic information than the dimension of the ambient vector space alone. They offer natural tools for distinguishing fuzzy labelings and suggest the possibility of a systematic classification theory for parameterized fuzzy vector structures.
The proposed framework opens several directions for future research. One natural problem is the study of suitable equivalence relations among fuzzy labelings and the characterization of canonical representatives within each equivalence class. It would also be interesting to investigate whether the dimension signature, or the collection of critical levels together with their corresponding dimension jumps, forms a complete invariant for important classes of fuzzy labelings.\\
Another promising direction is the investigation of reconstruction maps and level-dependent transformations. In particular, when a fuzzy labeling admits a representation of the form \[ x_\alpha=P(\alpha)x, \] where $P(\alpha)$ is a family of linear operators, the proposed theory may establish connections with parameter-dependent operator theory, linear time-varying systems, and control theory.\\
Finally, the present framework provides a natural foundation for extending classical functional analysis to the fuzzy setting. Possible developments include fuzzy Hilbert and Banach spaces, fuzzy linear operators, orthogonality, spectral theory, and infinite-dimensional parameterized fuzzy spaces. We hope that the ideas introduced in this paper will contribute to a deeper understanding of fuzzy linear structures and stimulate further research on the algebraic, geometric, and analytical aspects of fuzzy mathematics.

\bibliographystyle{elsarticle-num}
\bibliography{references}

\end{document}